\documentclass[11pt]{article}
\usepackage[latin1]{inputenc}
\usepackage[english]{babel}
\usepackage{amssymb}
\usepackage{amsmath,amsthm}
\usepackage[normalem]{ulem}
\usepackage{cite}

\newtheorem{thm}{Theorem}[section]

\newtheorem{cor}{Corollary}[section]

\newtheorem{rmk}{Remark}[section]
\newtheorem{exmp}{Example}[section]

\makeatletter

\let \al=\alpha
\let \be=\beta
\let \var=\varphi
\let \vare=\varepsilon

\let \de=\delta
\let \th=\theta
\let \la=\lambda

\let \q=\quad

\let \med=\medskip
\let \smal=\smallskip
\let \dps=\displaystyle
\let \ul=\underline
\let \ul=\underline
\let \ol=\overline

\let \ka=\kappa

\DeclareMathOperator{\e}{e}
\DeclareMathOperator{\diag}{diag}

\newcommand{\R}{\mathbb{R}}
\newcommand{\N}{\mathbb{N}}

\begin{document}


\begin{center}
	\textbf{\Large{Permanence for nonautonomous  differential systems with delays in the linear and nonlinear terms}}
	\end{center}

 \centerline{\scshape Teresa Faria}
 
 \smal
 \centerline{
 Departamento de Matem\'atica and CMAFCIO, Faculdade de Ci\^encias, Universidade de Lisboa}
  \centerline{
Campo Grande, 1749-016 Lisboa, Portugal}
 
\centerline{  Email:
teresa.faria@fc.ul.pt}

\vskip .5cm

\begin{abstract}  
In this paper, we obtain sufficient conditions for the permanence of a  family of  nonautonomous systems of delay differential equations. 
This family includes   structured models from mathematical biology, with  either discrete or distributed delays in both the linear and nonlinear terms, and where typically   the nonlinear terms are nonmonotone.
Applications to generalized Nicholson and  Mackey-Glass systems are given. 
\end{abstract}



\section{Introduction}
\setcounter{equation}{0}

In this paper, we investigate the persistence and permanence for a  class of multidimensional nonautonomous delay differential equations (DDEs), which includes a wide range of  structured models used in population dynamics,  neural networks,  physiological mechanisms,  engineering and many other fields.

We start by setting the abstract framework for the DDEs which we deal with in the next sections. For $\tau\ge 0$, consider the Banach space $C:=C([-\tau,0];\R^n)$   with the norm
$\dps\|\phi\|=\max_{\th\in[-\tau,0]}|\phi(\th)|$, where $|\cdot|$ is  a fixed norm in $\R^n$. We shall consider DDEs written in the abstract form
\begin{equation}\label{1.1}
x'(t)={\cal L}(t)x_t+f(t,x_t),\q t\ge t_0,
\end{equation}
where   $x_t\in C$ denotes  the segment of a solution $x(t)$ given by $x_t(\th)=x(t+\th), -\tau\le \th\le 0$, ${\cal L}(t):C\to\R^n$ is linear bounded and the nonlinearities are given by continuous functions  $f:[t_0,\infty)\times C\to [0,\infty)^n$. For simplicity, we set $t_0=0$. As  in many mathematical biology models, we shall  assume the existence and dominance of diagonal  linear  instantaneous negative feedback terms in  \eqref{1.1} and that
 each component $f_i$ of $f=(f_1,\dots,f_n)$ depends only on $t$ and on the  component $i$ of the solution: 
 \begin{equation}\label{1.3}
f(t,\phi)=(f_1(t,\phi_1),\dots, f_n(t,\phi_n))\q {\rm for}\q t\ge 0, \phi=(\phi_1,\dots,\phi_n)\in C.
\end{equation}


 Recently, there has been a renewed interest in questions of persistence and permanence for  DDEs. A number of  methods has been proposed to tackled different situations, depending on whether the equations are autonomous or not, scalar or multi-dimensional, monotone or nonmonotone.
See \cite{BB13,BB16,BBI,BIT,FariaAMC14,FariaJDDE16,FOS,FariaRost,GHM15,GHM18,GHM18a,ObayaSanz} and references therein, also for explanation of the models and motivation from  real world  applications.

Here,   the investigation concerning permanence  in
 \cite{FariaJDDE16,FOS} is pursued. In \cite{FariaJDDE16} only cooperative systems were considered, whereas in \cite{FOS} sufficient conditions for the  permanence of systems 
\begin{equation}\label{1.3}
x_i'(t)=-d_i(t)x_i(t)+\sum_{j=1}^n a_{ij}(t)x_j(t)+\sum_{k=1}^{m_i} \be_{ik}(t)    h_{ik}(t,x_i(t-\tau_{ik}(t))),\  i=1,\dots,n,\, t\ge 0,
\end{equation}
were established. Clearly, nonautonomous differential equations with multiple time-varying discrete delays are a  particular case of \eqref{1.3}. In this paper,  the more general framework of systems  \eqref{1.1} with (possibly distributed) delays in both ${\cal L}$ and $f$ is considered, although sharper results will be obtained for models of the form
$$x_i'(t)=-d_i(t)x_i(t)+\sum_{j=1}^n a_{ij}(t)x_j(t)+f_i(t,x_{i,t}),\  i=1,\dots,n,\, t\ge 0.
$$
The  criteria for permanence in  \cite{BBI,FariaAMC14,FOS} and many other works demand that all  the coefficients are bounded. More recently,  some authors have relaxed this restriction \cite{BB17,GHM15,GHM18,GHM18a}, though still under some boundedness requirements. Here, the boundedness of all the coefficients in \eqref{1.1} will not be a priori  assumed. We also emphasize that typically the nonlinearites $f_i(t,\phi_i)$ in \eqref{1.3} are not monotone in the second variable -- which is  the case of  Nicholson-type systems, for example. Nevertheless, some techniques for cooperative systems will be used. 
Our results  extend and improve some recent achievements in the  literature \cite{BBI,BIT,GHM18,GHM18a, Liu10,Tak1},
which mostly deal with scalar DDEs and/or cooperative $n$-dimensional models.

We now introduce some standard notation.
In what follows,  $\R^+=[0,\infty)$, the matrix $I_n$, or simply $I$, denotes the $n\times n$ identity matrix and  $\vec 1=(1,\dots, 1)\in\R^n.$ For $\tau>0$, the set
$C^+=C([-\tau,0];(\R^+)^n)$ is the cone of nonnegative functions in $C$ and $\le $  the usual partial order generated by  $C^+$: $\phi\le \psi$ if and only if $\psi-\phi\in C^+$.  A vector $v\in \R^n$  is identified in $C$ with the constant function $\psi(\th)=v$ for $-\tau\le \th\le 0$. For $\tau=0$,  we take $C=\R^n,  C^+=[0,\infty)^n$; a vector $v\in\R^n$ is positive  if all its components are positive, and we write $v>0$. We write $\phi< \psi$ if  $\psi(\th)<\phi(\th)$ for $\th\in [-\tau,0]$; the relations $\ge$ and $> $ are also defined in the usual way. 

For nonlinear DDEs \eqref{1.1}  under conditions of existence and uniqueness of solutions, $x(t,\sigma,\phi)$ denotes the solution of \eqref{1.1} with initial condition $x_\sigma=\phi$, for $(\sigma,\phi)\in \R^+\times C$. For models inspired by mathematical biology applications, we shall consider 
 $$C_0^+=\{\phi\in C^+:\phi(0)>0\}$$
 as the set of admissible initial conditions. Without loss of generality, we shall restrict the analysis to solutions $x(t,0,\phi)$ with $\phi \in C_0^+$, and assume that $f$ is sufficiently regular so that such solutions are defined on $\R^+$.
  If the set $C_0^+$ is (positively) invariant for \eqref{1.1}, the notions of (uniform) persistence, permanence and stability always refer to solutions with initial conditions in $C_0^+$.  In this way,
  we say that the system is {\bf uniformly persistent} (in $C_0^+$) if there exists a positive uniform lower bound  for all solutions with initial conditions in $C_0^+$; i.e., there is $m>0$ such that all solutions $x(t)=x(t,0,\phi)$ with $\phi \in C_0^+$ are defined on $\R^+$ and satisfy  $x_i(t,0,\phi)\ge m$ for $t\gg 1$ and $i=1,\dots, n$.  The system  \eqref{1.1} is said to be {\bf permanent} if there exist positive constants $m,M$  such that  all solutions $x(t)=x(t,0,\phi)$ with $\phi \in C_0^+$ are defined on $\R^+$ and satisfy $m\le x_i(t)\le M$ for $t\gg 1$ and $i=1,\dots, n$. As usual, the expression $t\gg 1$  means  ``for $t>0$ sufficiently large". 
   For short, here we say that a DDE $x'(t)=F(t,x_t)$ is {\bf cooperative} if  $F=(F_1,\dots,F_n)$ satisfies the {\it quasi-monotone condition} (Q) in \cite{Smith}: if $\phi,\psi\in C^+$ and $\phi\ge \psi$, then $F_i(t,\phi)\ge F_i(t,\psi)$ for $t\ge 0$, whenever $\phi_i(0)=\psi_i(0)$ for some $i$.
  
  The remainder of this paper is divided into three sections. In Section 2, we establish sufficient conditions for the uniform persistence and permanence for a large family of nonlinear system \eqref{1.1}.  To illustrate the results, generalized Nicholson and Mackey-Glass systems are considered  in Section 3, together  with examples, as well as  counter-examples showing the necessity of some hypotheses. The paper ends with a short section of conclusions and open problems.  
  
\section{Persistence and permanence for a class of  nonautomous DDEs} 
\setcounter{equation}{0}

In this section, we establish  explicit and easily verifiable criteria  for both the persistence and the   permanence of systems  \eqref{1.1} with nonlinearities $f$ expressed by \eqref{1.3}. 

Let $C:=C([-\tau,0];\R^n)$  with the supremum norm be the phase space. We start with  a general    nonutonomous  linear   differential equation in $C$,
\begin{equation}\label{Lin0}
x'(t)={\cal L}(t)x_t,
\end{equation}
where  ${\cal L}:\R\to L({\cal C},\R^n)$, $L(C,\R^n)$ is the usual space of bounded linear operators from $C$ to $\R^n$ equipped with  the operator norm,  and $t \mapsto {\cal L}(t)\phi$  is Borel measurable for each $\phi$, with $\|{\cal L}(t)\|$   bounded on $\R^+$ by a  function $m(t)$ in $L_{\rm loc}^1(\R^+;\R)$.

Assuming the exponential asymptotic stability of \eqref{Lin0}, next theorem provides conditions for the dissipativeness and  extinction of  perturbed nonlinear systems. Its proof is easily deduced from  the variation of constant formula \cite{HaleLunel} and  arguments similar to the ones for ODEs, thus it is omitted. 

\begin{thm}\label{thm2.2}  Assume that the system \eqref{Lin0} is exponentially asymptotically stable, and
consider the perturbed equation
\begin{equation}\label{LinP}
x'(t)={\cal L}(t)x_t+f(t,x_t),\q t\ge 0,
\end{equation}
 where $f:[0,\infty)\times S\to \R^n$ is continuous and $S$ is a (positively)  invariant set for \eqref{LinP}. \\
 (i) If  $f$ is bounded, 
 then  \eqref{LinP}
is dissipative; i.e.,  all solutions of \eqref{LinP}
are defined  on $[0,\infty)$ and there exists $M>0$ such that any solution $x(t)$ of \eqref{LinP} satisfies
$ \limsup_{t\to\infty}|x(t)|\le M.$\\
(ii) If   there exists $\be:\R^+\to\R^+$ measurable with $\int^{\infty}\be (s)\, ds<\infty$ such that
$|f(t,\phi)|\le \be(t)\|\phi\|,\ t\gg 1,$
then all solutions $x(t)$  of \eqref{LinP} satisfy $\limsup_{t\to\infty}x(t)=0.$
\end{thm}

For \eqref{Lin0}, we now suppose that  ${\cal L}=({\cal L}_1,\dots,{\cal L}_n)$ is given by
\begin{equation}\label{1.2}
{\cal L}_i(t)\phi=-d_i(t)\phi_i(0)+\sum_{j=1}^nL_{ij}(t)\phi_j,\q  t\ge 0, \phi=(\phi_1,\dots,\phi_n)\in C,  i=1,\dots,n,
\end{equation}
with $d_i(t)>0$ and $ L_{ij}(t)$  bounded  linear functionals. Although it is not relevant for our results, we may assume that   $L_{ii}(t)$  is non-atomic at zero (see \cite{HaleLunel} for a definition). For \eqref{Lin0}, define 
the $n\times n$ matrix-valued functions 
 \begin{equation}\label{D&A}
  D(t)=\diag \,  (d_1(t),\dots,d_n(t)),\q  A(t)=\Big [a_{ij}(t)\Big ],
 \end{equation}
where $$a_{ij}(t)=\|L_{ij}(t)\|,\q t\ge 0,\, i,j\in\{1,\dots,n\}.$$   

For \eqref{Lin0}, the  general hypotheses below will be considered:
  \begin{itemize}
\item[(H1)] the functions $d_{i}:[0,\infty)\to (0,\infty), L_{ij}:[0,\infty)\to L(C([-\tau,0],\R),\R)$ are continuous   (for some $\tau\ge 0$),   $i,j=1,\dots,n$;
\item[(H2)]  there exist a vector $v> 0$ and a constant $\de>0$ such that $\Big [D(t)-A(t)-\de I_n\Big ]v\ge 0$ for  $t\gg 1$.
\end{itemize}

Instead of (H2), one may assume: 
 \begin{itemize}
\item[(H2*)]  
 there exist a vector $v> 0$ and a constant $\al>1$ such that $D(t)v\ge \al A(t)v $ for  $t\gg 1$ .
\end{itemize}


 With the notation in \eqref{D&A}, e.g.  assumption (H2) above translates as: there exist a vector $v=(v_1,\dots,v_n)> 0$ and $T\ge 0,\de>0$ such that $d_i(t)v_i-\sum_{j=1}^na_{ij}(t)v_j\ge \de v_i$ for all $t\ge T,\, i=1,\dots,n$.

 Next theorem gives some  stability results  selected from \cite{Faria20}.

 \begin{thm}\label{thm2.1}  Consider system \eqref{Lin0} under (H1), and assume one of the following sets of conditions:
  \vskip 0mm
  (i) (H2) is satisfied and $a_{ij}(t)$ are bounded functions on $\R^+$ for all $i,j=1,\dots,n$;
\vskip 0mm
(ii)  (H2*) is satisfied and
$\liminf_{t\to\infty}d_{i}(t)>0$ for $ i=1,\dots,n$;\vskip 0mm
(iii)   \eqref{Lin0} is the ODE system $x_i'(t)=-d_i(t)x_i(t)+\sum_{j\ne i}d_{ij}(t)x_j(t),1\le i\le n$, and  (H2) is satisfied with $a_{ij}(t)=|d_{ij}(t)|$.\vskip 0mm
Then, \eqref{Lin0} is exponentially asymptotically stable; in other words,  there exist $k,\al>0$ such that 
$$|x(t,t_0,\var)|\le k\e^{-\al (t-t_0)}\|\var\|\q {\rm for\ all \ } t\ge t_0\ge 0,\var \in C.$$
\end{thm}
\begin{proof} The result follows from the criteria in \cite[Theorem 3.1]{Faria20}.
\end{proof}

Henceforth,  we consider delay differential systems written as
\begin{equation}\label{NonLin}
x_i'(t)=-d_i(t)x_i(t)+\sum_{j=1}^n L_{ij}(t)x_{j,t}
+f_i(t,x_{i,t}), \q t\ge0,  i=1,\dots,n,
\end{equation}
with  the linear functionals  $L_{ij}(t)$  {\it nonnegative} (i.e.,  $L_{ij}(t)\phi_j\ge 0$ for $\phi_j\ge 0$) and  
 $f_i(t,\phi_i)$  continuous and  satisfying same requirements  formulated below.
 Recall that, by the Riesz representation theorem, the nonnegative bounded functionals $L_{ij}(t)$ have a representation 
 \begin{equation}\label{Lij}
 L_{ij}(t)\phi_j =a_{ij}(t) \int_{-\tau}^0 \phi_j(s) d_s\nu_{ij}(t,s),
 \end{equation} 
 where $a_{ij}(t)=\|L_{ij}(t)\|$, the functions $\nu_{ij}(t,s)$ are defined for $(t,s)\in \R^+\times [-\tau,\infty)$,  are continuous in $t$, left-continuous and {\it nondecreasing}  in  $s$,  and normalized so that
$\int_{-\tau}^0 d_s\nu_{ij}(t,s)=1,\ t\ge 0.$
 In the case of no delays in \eqref{Lij}, then  $L_{ij}(t)x_{j,t}=a_{ij}(t)x_j(t)$ with $a_{ij}(t)\ge 0$.
 Clearly, this framework includes the particular case of  DDEs with multiple  time discrete delays:
  \begin{equation}\label{3_disc}
\begin{split}
x_i'(t)=&-d_i(t)x_i(t)+\sum_{j=1}^n \sum_{p=1}^{n_j}a_{ijp}(t)x_j(t-\sigma_{ijp}(t))\\&+\sum_{k=1}^{m_i} \be_{ik}(t)    g_{ik}(t,x_i(t-\tau_{ik}(t))),\q i=1,\dots,n,\ t\ge 0,
\end{split}
\end{equation}
where the coefficients and delays  are all continuous and nonnegative.

Systems \eqref{NonLin}  are sufficiently general  to encompass many relevant models from mathematical biology and other fields. 
In some contexts, they are interpreted as  structured  models for
 populations distributed over $n$ different  classes or patches, with migration among the patches, where $x_i(t)$ is the density of the species  on  class $i$,
$a_{ij}(t)$ ($j\ne i)$ is the migration coefficient from class $j$ to class $i$,
$d_i(t)$ the coefficient of instantaneous loss for class $i$, and $f_i(t,\phi_i)$ is the   birth function for  class $i$.   Although most models do not include delays in the migration terms, structured models where delays intervene in the linear terms have deserved the attention of a number of researchers, see e.g. Takeuchi et al. \cite{Tak1}.  
 We also observe that in biological models most situations  require  a single delay for each population, however multiple or distributed delays naturally appear in  neural networks models, or generalizations of the classic Mackey-Glass equation used as  hematopoiesis models. We refer the reader to \cite{BB13,BB16,BB17,BBI,BIT,GHM18,Smith,Tak1}, for  real interpretation of the DDEs under consideration and some applications.

In what follows, for $\phi=(\phi_1,\dots,\phi_n)\in C^+$ we use the notation
$$\ul{\phi_i}=\min_{s\in [-\tau,0]} \phi_i(s).$$

To establish the permanence of  \eqref{NonLin}, we further impose that the nonlinearities satisfy the following conditions:

%

	\begin{itemize}
\item[ (H3)] the functions $f_i:\R^+\times C^+\to\R^+$ are completely continuous and locally Lipschitzian in the second variable, $i\in \{1,\dots,n\};$
\item[(H4)]  there exist continuous functions  $\be_i:\R^+\to (0,\infty),h_i^-:\R^+\to \R^+$,  
with $ h_i^-(0)=0, h_i^-(x)>0$ for $x>0$ and with right-hand derivative at zero $(h_i^-)'(0+)=1$, such that,  for $i\in \{1,\dots,n\}$,
\begin{equation}\label{autonBound}
f_i(t,\phi_i)\ge \be_i(t) h_i^-(\ul{\phi_i}),\q t\gg 1\ {\rm and}\ \phi_i\in C^+([-\tau,0];\R^+); 
 \end{equation}
 \item[(H5)]  there exist a positive vector $v$ and a constant $\de >0$ such that $[M(t)-\de I_n]v\ge 0,\ t\gg 1$, where $M(t)$ is the matrix-valued function defined by
\begin{equation}\label{M(t)}
\begin{split}
M(t)&=B(t)+A(t)-D(t),\q {\rm where}\\
B(t)&=\diag \, (\be_1(t),\dots,\be_n(t)),\q t\ge 0,
\end{split}
\end{equation}
 with
 $D(t),A(t)$  as above and $\be_i(t)$ as in (H4). \end{itemize}

 Instead  of  (H5), we shall often assume:
 \begin{itemize}
\item[(H5*)]  there exist a vector $v> 0$ and a constant $\al>1$ such that $ B(t)v\ge\al \Big[D(t)-A(t)\Big ]v$ for  $t\gg 1$.
\end{itemize}

 Some comments about these assumptions are given in the remarks below.

 \begin{rmk}\label{rmk2.2} {\rm  
If the coefficients $\be_i(t)$ are {\it bounded}, then (H5) implies (H5*).
Indeed, if (H5) holds and  there exists $M=\max_{1\le i\le n}\sup_{t\ge 0} \be_i(t)v_i$  (as a matter of fact, it suffices that $d_i(t)v_i-\sum_j a_{ij}(t)v_j\le M$ for some $M$),  then (H5*) is satisfied with the same $v>0$ and any $\al\in(0, 1+\de/M]$. The converse is also true if $\be_i(t)$ are all {\it bounded from below}  by a positive constant, since in this case (H5*) implies that (H5) is satisfied with the same $v>0$ and $\de\in (0,(1-\al^{-1})c]$ for $c>0$ such that $c=\min_{1\le i\le n}\inf_{t\ge 0}\be_i(t)v_i$. Similarly, one easily verifies (conf. \cite{Faria20}) that  when  the coefficients $d_i(t)$ are bounded from below by a positive constant, then (H2*) implies (H2); and that,  if $a_{ij}(t)$ are all bounded, then (H2) implies (H2*).

In the study of stability for nonautonous  DDEs, a condition as (H2*) with $v=\vec 1$ has been often presented (see e.g. \cite{FariaJDDE16,GHM18})  in the equivalent form  
$\liminf_{t\to \infty} \frac{d_i(t)}{\sum_{j=1}^na_{ij}(t)}>1.$
Analogously, (H5*) with $v=\vec 1$ can be written as
$\liminf_{t\to \infty} \frac{\be_i(t)}{d_i(t)-\sum_{j=1}^na_{ij}(t)}>1.$
 }\end{rmk}

 \begin{rmk} \label{rmk2.3} {\rm If \eqref{autonBound} holds with a function $h_i^-$ satisfying $h_i^-(0)=0$, $h_i^-(x)>0$ on $(0,\infty)$ and $(h_i^-)'(0+)=c_i>0$, by replacing $h_i^-(x),\be_i(t)$ by $\bar h_i^-(x)=c_i^{-1}h_i^-(x),\bar \be_i(t)=c_i\be_i(t)$, respectively, we may always assume that $(h_i^-)'(0+)=1$.  }
 \end{rmk}

 \begin{rmk}\label{rmk2.4}  {\rm For \eqref{NonLin} under the above hypotheses, rescaling the variables by $\hat x_i(t)=v_i^{-1}x_i(t)$ $ (1\le i\le n)$, where $v=(v_1,\dots,v_n)> 0$ is a vector as in (H5) or (H5*), we obtain a new system
\begin{equation}\label{3'}
\hat x_i'(t)=- d_i(t)\hat x_i(t)+\sum_{j=1}^n \hat L_{ij}(t)\hat x_{j,t}+\hat f_i(t,\hat x_{i,t}),\  i=1,\dots,n,\ t\ge 0,
\end{equation}
where $ \hat a_{ij}(t):=\|\hat L_{ij}(t)\|= v_i^{-1}a_{ij}(t) v_j$ and
$\hat f_i(t,\phi_i)=v_i^{-1}f_i(t,v_i\phi_i)$ satisfy (H3)-(H4), with $h_i^-(x)$ replaced by $\hat h_i^-(x)=v_i^{-1}h_i^-(v_ix)$. In this way, and after dropping the hats for simplicity, we may consider an original system \eqref{NonLin} and  take $v=\vec 1$  in (H5)  or (H5*).}\end{rmk}

 %


The main criterion for the permanence of \eqref{NonLin} is now established.

\begin{thm}\label{thm2.3} For  \eqref{NonLin}, assume   (H1)--(H4). Furthermore, let the following conditions hold:\vskip 0mm
 (i) either $L_{ij}(t)\phi_j=a_{ij}(t)\phi_j(0)$  with $a_{ij}(t)\ge 0$, for all $ i,j=1,\dots,n$ and $t\ge 0$ (in other words, there are no delays in  \eqref{Lin0}), or  $L_{ij}(t)$ are  nonnegative and  $a_{ij}(t)=\|L_{ij}(t)\|$  are bounded on $\R^+$, $i,j=1,\dots,n$;\vskip 0mm
 (ii) either (H5) is satisfied  and  $\dps\limsup_{t\to\infty}\be_i(t)<\infty$ for all $i$, or (H5*) is satisfied and $\dps\liminf_{t\to\infty}\be_i(t)>0$ for all $i$;\vskip 0mm
  (iii)  $f_i(t,\phi_i)$ are bounded,  $ i=1,\dots,n$.\\
Then   system  \eqref{NonLin}  is permanent (in $C_0^+$).\end{thm}

\begin{proof} The proof follows along the main ideas in \cite[Theorem 3.3]{FOS}, however new arguments are used to take into account  the more general form of \eqref{NonLin}, that  delays are allowed in the linear part and that the coefficients $d_i(t)$ are not required to be bounded -- as well as $a_{ij}(t)$, if there are no delays in $L_{ij}(t)$.
\smal


{\it Step 1}.  Writing  \eqref{NonLin} as $x'(t)=F(t,x_t)$, it is clear that $F$ is continuous, locally Lipschitzian in the second variable and bounded on bounded sets of $\R^+\times C^+$. From Theorem \ref{thm2.1}, \eqref{Lin0} is exponentially asymptotically stable  (for the case of no delays in the linear functionals $L_{ij}(t)$, recall that the boundedness of $a_{ij}(t)$ is not required). Theorem \ref{thm2.2} implies that  \eqref{NonLin} is dissipative.
Observe that the solutions of \eqref{NonLin} satisfy the ordinary differential inequalities $ x_i'(t)\ge - d_i(t)x_i(t)\ (1\le i\le n)$, thus the solutions $x(t,0,\phi)$ with $\phi\in C_0^+$ are positive for $t\ge 0$.

From (ii) and Remark \ref{rmk2.2},  both (H5) and (H5*) are satisfied, with a common vector $v>0$.
 By the scaling described  in Remark \ref{rmk2.4}, without loss of generality we may take $v=\vec 1$ in (H5), (H5*). Thus,    there are $\de_0>0$, $ \al_0>1$ and $T_0>0$ such that \begin{equation}\label{h5'}
\begin{split}
\be_i(t)&\ge d_i(t)-\sum_j a_{ij}(t)+\de_0,\\
\be_i(t)&\ge \al_0\Big (d_i(t)-\sum_j a_{ij}(t)\Big),\q t\ge T_0,i=1,\dots,n.
\end{split}
\end{equation}
Summing up these inequalities, we obtain that
\begin{equation}\label{h5}
\begin{split} 
\al^{-1}\be_i(t)\ge \Big (d_i(t)-\sum_j a_{ij}(t)\Big)+\de\,\q t\ge T_0,i=1,\dots,n,
\end{split}
\end{equation}
for $\al=(1+\al_0)/2>1, \de_0=\de/2$.

 Choose $M>0$ such that any positive solution $x(t)=x(t,0,\phi)$ of \eqref{NonLin}
 satisfies $0< x_i(t)\le M$, for all $i$ and $t\ge T$, with $T=T(\phi)$ sufficiently large.  Next,
choose $m>0$ such that  $h_i^-(x)$ is strictly increasing 
 with  $\al^{-1}x<h_i^-(x)$  on the interval $(0,m]$, and
  $h_i^-(m)=\min_{x\in [m,M]}h_i^-(x)$ for  all $i$. 
    In this way,  $f_i(t,x_{i,t})\ge \be_i(t) H_i\big (\dps\min_{-\tau\le s\le 0}x_i(t+s)\big)$  for any solution $x(t)$ of \eqref{NonLin},  for all $ i$ and $t\ge T_1:=\max \{T(\phi),T_0\}$, where $H_i(x)=\begin{cases}h_i^-(x)\q {\rm if}\q 0\le x\le m\\
h_i^-(m)\q {\rm if}\q x\ge m\end{cases}$. Replacing $h_i^-(x)$ by the function $\min\{h_i^-(x),x\}$, we may also assume that $H_i(x)\le x$ for all $x\ge 0$.
Note  that $H_i$ are monotone functions.
  We now compare  \eqref{NonLin} from below with the auxiliary  cooperative system:
\begin{equation}\label{Coop}
\begin{split}
x_i'(t)&=-d_i(t)x_i(t)+\sum_{j=1}^n L_{ij}(t)x_{j,t}\\
&+\be_i(t)  H_i\big (\dps\min_{-\tau\le s\le 0}x_i(t+s)\big)
=:G_i(t,x_t),\q  i=1,\dots,n.
\end{split}
\end{equation}
  From \cite{Smith}, this implies that $x(t,T_1,\phi,F)\ge x(t,T_1,\phi,G)$, where $x(t,T_1,\phi,F)$ and $ x(t,T_1,\phi,G)$ are the solutions of \eqref{NonLin} and \eqref{Coop} with initial condition $x_{T_1}=\phi\in C_0^+$, respectively.
Therefore,  if \eqref{Coop} is uniformly persistent,  \eqref{NonLin} is uniformly persistent as well.
 
  We now derive the uniform persistence of \eqref{Coop} by showing that, for any solution $x(t)=x(t,t_0,\phi,G)$ of \eqref{Coop},
 there exists $T\ge t_0$ such that
 \begin{equation}\label{m}
  x_i(t)\ge m\q {\rm for\ all}\q  t\ge T,1\le i\le n.
\end{equation}
This is   proven in several steps.

\med

{\it Step 2}.    
We first prove that  the ordered interval $[m,\infty)^n:=\{\phi\in C: \phi_i\ge m,\, i=1,\dots,n\}\subset C$ is  invariant for  \eqref{Coop}  for $t\in[T_0,\infty)$. 

 Note that both the functions $H_i(x)$ and the operators $L_{ij}(t)$ are nondecreasing   and $H_i(x)=h_i^-(x)> \al^{-1}x$ on $(0,m]$. If $\phi\in [m,\infty)^n$ and $\phi_i(0)=m$ for some $i$, from \eqref{h5'}  we therefore obtain, for $t\ge T_0$,
$$G_i(t,\phi)\ge m\Big [-d_i(t)+\sum_j a_{ij}(t)\Big]+\be_i(t) H_i(m)> m\Big [-d_i(t)+\sum_j a_{ij}(t)+\al^{-1}\be_i(t)\Big]\ge 0.$$
From  \cite[Remark 5.2.1]{Smith},  it follows that the set $[m,\infty)^n\subset C$ is positively  invariant for \eqref{Coop}.
%

%
%
%
%

\smal

{\it Step 3}.  
For  $ T_0$ as before, define  $$s_0=\min_j\min_{t\in [T_0,T_0+\tau]} x_j(t),\q s_1=\min_j\min_{t\in[T_0+\tau,T_0+2\tau] }x_j(t).$$
Let $s_1=x_i(t_0)$, for some $t_0\in [T_0+\tau, T_0+2\tau]$ and $i\in \{1,\dots,n\}$.
We now  show that  $ s_1< m$ implies that $s_1>s_0$.  

If $s_1\le s_0$, then
 $$s_1=x_i(t_0)=\min_j\min_{t\in[T_0,T_0+2\tau] }x_j(t).$$
Assuming that $s_1\le s_0$ and $x_i(t_0)<m$,   since $x_{j,t}\ge x_i(t_0)$ for $t\in [t_0-\tau,t_0]$, we get
\begin{equation*}
\begin{split}
0\ge x_i'(t_0)
\ge &\Big (-d_i(t_0)+\sum_j a_{ij}(t_0)\Big)x_i(t_0)+\be_i(t_0)H_i(x_i(t_0))\\
=&\Big (-d_i(t_0)+\sum_j a_{ij}(t_0)\Big)x_i(t_0)+\be_i(t_0)h_i^-(x_i(t_0))\\
>&\Big (-d_i(t_0)+\sum_j a_{ij}(t_0)+\al^{-1}\be_i(t_0)\Big)x_i(t_0)\ge 0,
\end{split}
\end{equation*}
which is a contradiction.
 This shows that $s_1>s_0$ whenever $ s_1< m$. 
 
 {\it Step 4}. 
Define the  sequence 
$$  s_k=\min_j\min_{t\in[T_0+k\tau,T_0+(k+1)\tau] }x_j(t),\q k\in \N_0.$$ 

For the sake of contradiction,
assume that $s_k<m$ for all $k\in\N_0$.   Thus, reasoning as in Step 3,   $(s_k)$ is strictly increasing.
Let $t_k\in I_k:=[T_0+k\tau,T_0+(k+1)\tau]$ be such that $s_k=x_{i_k}(t_k)$, for some $i_k\in\{1,\dots,n\}.$ By jumping some of the  intervals $I_k$ and considering a subsequence of $(t_k)$, still denoted by $(t_k)$, we may consider a unique $i\in\{1,\dots,n\}$ such that $s_k=x_i(t_k)$. Denote $\ell=\lim_k s_k>0$. 


%
%

Let $\al>1$ and $\de>0$ as in \eqref{h5}. We now claim that
\begin{equation}\label{s2-claim}
s_{k+1}\ge  \al   \min_j  H_j(s_k),\q k\in \N_0.
\end{equation}
Otherwise, suppose that there is $k$  such that $s_{k+1}<\al \min_j H_j( s_k)$.

 We distinguish two situations: either there are no delays in  \eqref{Lin0}, or $a_{ij}(t)$ are all bounded in $\R^+$ -- in which case we suppose that $\al$ is chosen so that it also satisfies $1<\al<\frac{M_1}{M_1-\de},$ where
 $\max_{1\le j\le n}\sup_{t\ge 0}(\sum_j a_{ij}(t))\le M_1$. 
 
%

 First, we treat the case of no delays in the linear part ${\cal L}$ of \eqref{Lin0}. In this situation, $ L_{ij}(t)x_{j,t}= a_{ij}x_j(t)\ge  a_{ij}(t)s_{k}$ for $t\in I_{k}\, (k\in \N_0)$.
Estimate \eqref{h5'} leads to
\begin{equation*}
\begin{split}
0\ge x_i'(t_{k+1})\ge &-d_i(t_{k+1})s_{k+1}+\sum_j a_{ij}(t_{k+1})s_{k+1}+\be_i(t_{k+1})H_i(s_k)\\
>&\Big (-d_i(t_{k+1})+\sum_j a_{ij}(t_{k+1})+\al^{-1}\be_i(t_{k+1})\Big )s_{k+1}\ge 0,\\
\end{split}
\end{equation*}
which is not possible. Thus,  \eqref{s2-claim} holds.

When  delays are allowed in the linear part,  we can write  $ L_{ij}x_{j,t}\ge  a_{ij}(t)s_k$ for $t\in I_{k+1}$, and  $s_k\ge h^-_i(s_k)=H_i(s_k)> \al^{-1}s_{k+1}$, thus we have
\begin{equation*}
\begin{split}
0\ge x_i'(t_{k+1})\ge &-d_i(t_{k+1})s_{k+1}+\sum_j a_{ij}(t_{k+1})s_{k}+\be_i(t_{k+1})H_i(s_k)\\
> &\Big(\al^{-1}\be _i(t_{k+1})+\sum_j a_{ij}(t_{k+1})-d_i(t_{k+1})\Big)s_{k+1}+(s_k-s_{k+1})\sum_j a_{ij}(t_{k+1})\\
\ge & \de s_{k+1}+s_{k+1}(\al^{-1}-1)M_1=[\de +(\al^{-1}-1)M_1]s_{k+1}>0,
\end{split}
\end{equation*}
which is not possible. Thus, claim \eqref{s2-claim} is proven.

From \eqref{s2-claim}, we obtain
$m\ge \ell\ge \al \min_j H_j(\ell)=\al \min_j h^-_j(\ell)>\ell$,
which is not possible. Therefore, $s_k\ge m$ for some $k$, and the result follows by Step 2.
%
%
%
\end{proof}

\begin{rmk}\label{rmk2.5} {\rm 
For $v=(v_1,\dots,n)$ and $m>0$  as in the above proof, one concludes that any solution $x(t)=x(t,0,\phi)\ (\phi\in C_0^+)$ of   \eqref{NonLin} satisfies
$ \min_{t\ge T} x_j(t)\ge mv_j\, (1\le j\le n)$ for some $T=T(\phi)$.}
\end{rmk}

It is clear that 
assumption (H2) was used in the above proof only to derive that \eqref{NonLin} is dissipative. In the case of bounded nonlinearities, Theorem \ref{thm2.2} shows that Theorem \ref{thm2.3} is still valid if one replaces (H2) by the requirement of having  \eqref{Lin0} exponentially asymptotically stable, as stated below.

\begin{thm}\label{thm2.3'} For  \eqref{NonLin}, assume   (H1), (H3), (H4)  and that:\vskip 0mm
(i) the linear system $x_i'(t)=-d_i(t)x_i(t)+\sum_{j=1}^n L_{ij}(t)x_{j,t},\ i=1,\dots,n$,
is exponentially asymptotically stable;\vskip 0mm
 (ii) either $L_{ij}(t)\phi_j=a_{ij}(t)\phi_j(0)$  with $a_{ij}(t)\ge 0$, for all $ i,j=1,\dots,n$ and $t\ge 0$, or  $L_{ij}(t)$ are  nonnegative and  $a_{ij}(t)=\|L_{ij}(t)\|$  are bounded on $\R^+$, $i,j=1,\dots,n$;\vskip 0mm
  (ii) either (H5) is satisfied  and  $\dps\limsup_{t\to\infty}\be_i(t)<\infty$ for all $i$, or (H5*) is satisfied and $\dps\liminf_{t\to\infty}\be_i(t)>0$ for all $i$;\vskip 0mm
 (iv) $f_i(t,x)$ are bounded, $i=1,\dots,n, k=1,\dots,m_i$.\\
Then   system  \eqref{NonLin}  is permanent (in $C_0^+$).\end{thm}

On the other hand, if the nonlinear terms $f_i$ in \eqref{NonLin} are not bounded but are sublinear, a condition stronger than (H2) still gives the dissipativeness of the system.
   
   \begin{thm}\label{thm2.6} Consider  \eqref{NonLin}, under  (H1),~(H3). Suppose that there exist functions $\be_i^+,h_i^+:\R^+\to (0,\infty)$, a constant $R>0$ and a  vector $u>0$ such that, for $\phi=(\Phi_1,\dots,\phi_n)\in C_0^+$ and $i=1,\dots,n$:\vskip 0mm
 (i) $0\le f_i(t,\phi_i)\le \be_i^+(t)h_i^+(\|\phi_i\|)$ for $t\gg 1$ and $\|\phi_i\|\ge R$;\vskip 0mm
   (ii) $\limsup_{x\to\infty}\frac{h_i^+(x)}{x}<1$;\vskip 0mm
   (iii) $\Big [D^+(t)-A(t)\Big ]u\ge 0$ for $t\gg 1$, where $D^+(t)=\diag(d_1(t)-\be_1^+(t),\dots,d_n(t)-\be_n^+(t))$.\\
Then     \eqref{NonLin}  is dissipative.\end{thm}

\begin{proof} Once again, after a scaling, we may consider that the positive vector $u$ in (iii) is the unit vector $\vec 1$. From (ii), take $R_0\ge R$ such that $h_i^+(x)< x$ for $x\ge R_0$. Let $x(t)$ be a  solution with initial contition in $C_0^+$. We claim that $\limsup_{t\to\infty}|x(t)|\le R_0$. Otherwise, there is $i$ and a sequence $t_k\to\infty$ such that $\|x_{t_k}\|=x_i(t_k)>R_0$ and $x_i'(t_k)\ge 0$. This would imply
\begin{equation*}
\begin{split}
x_i'(t_k)&\le-d_i(t_k)x_i(t_k)+\sum_j a_{ij}(t_k)\|x_{j,t_k}\|+\be_i^+(t_k)h_i^+(\|x_{i,t_k}\|)\\
&<-x_i(t_k)\big[d_i(t_k)-\be_i^+(t_k)+\sum_j a_{ij}(t_k)\big ]\le 0,
\end{split}
\end {equation*}
which is not possible.
\end{proof}

The  previous arguments also allows us to derive sufficient conditions for the uniform persistence of  \eqref{NonLin}  without requiring that the system is dissipative, nor that the coefficients $\be_i(t)$ are bounded. 

\begin{thm}\label{thm2.4} For  \eqref{NonLin}, assume  (H1), (H3), (H4)  and (H5*), and  the following conditions:\vskip 0mm
 (i) either $L_{ij}(t)\phi_j=a_{ij}(t)\phi_j(0)$  with $a_{ij}(t)\ge 0$, for all $ i,j=1,\dots,n$ and $t\ge 0$, or  $L_{ij}(t)$ are  nonnegative and  $a_{ij}(t)=\|L_{ij}(t)\|$  are bounded on $\R^+$, $i,j=1,\dots,n$;\vskip 0mm
   (ii) $\liminf_{x\to\infty}\be_i(t)>0$ and $\liminf_{x\to\infty}h_i^-(x)>0$  for $h_i^-(x)$ as in (H4), $i=1,\dots,n$.\\
Then   system  \eqref{NonLin}  is uniformly persistent.\end{thm}

\begin{proof} With the notations in the above proof,  choose $m>0$ such that $h_i^-(x)$ is strictly increasing on $(0,m]$
 and  $\al^{-1}x<h_i^-(x)$ for  all $i$ and $x\in (0,m]$.
From (H4) and (ii), there exists $M>0$ such that $h_i^-(x)\ge M$ for all $i$ and $t\ge m$. If necessary, find $m_0\in (0,m)$, still denoted by $m$, such that $h_i^-(m)\le M$  and take
   $H_i(x)$ as  in \eqref{Coop}. 
Since $\liminf_{x\to\infty}\be_i(t)>0$, then (H5) also holds (conf.~Remark \ref{rmk2.2}).  
 The conclusion follows as in Theorem \ref{thm2.3}.\end{proof}

When the linearities do not have delays,  the above proof  only requires  the use of assumption (H5*) to show the uniform persistence, but not of (H5). This observation and Theorem \ref{thm2.1}(iii) allow us to conclude the following:
\begin{cor}\label{cor2.1} For  
\begin{equation}\label{NonLinFOS}
x_i'(t)=-d_i(t)x_i(t)+\sum_{j=1}^n a_{ij}(t)x_j(t)+f_i(t,x_{i,t}), \ i=1,\dots,n,
\end{equation}
assume that  (H1), (H3), (H4)  and (H5*) are satisfied, with $a_{ij}(t)\ge 0$  on $\R^+$.\\
(a) If (H2) holds and
 $f_i(t,x)$ are bounded, for all $i,k$,
 then     \eqref{NonLinFOS}  is permanent.\\
(b) If  $\liminf_{x\to\infty}h_i^-(x)>0$ for $h_i^-(x)$ as in (H4),  for all $i$,
   then  \eqref{NonLinFOS}  is uniformly persistent.\end{cor}

   We end this section with two remarks, leading to  more precise and general results.

\begin{rmk}\label{rmk2.0} {\rm More explicitly,   we could have written the linear DDE \eqref{Lin0}  as
\begin{equation*}
\begin{split}
x_i'(t)&=-d_i(t)x_i(t)+\sum_{j\ne i} d_{ij}(t)x_j(t)+\sum_{j=1}^n a_{ij}(t) \int_{-\tau}^0 x_j(t+s) d_s\nu_{ij}(t,s),\ 1\le i\le n,
\end{split}
\end{equation*}
with $d_i(t)>0, d_{ij}(t)\ge 0\, (j\ne i)$ and $a_{ij}(t),\nu_{ij}(t,s)$ as above, with $s\mapsto \nu_{ij}(t,s)$ non atomic at zero, and apply more precise criteria for its exponential asymptotic stability, see  \cite{Faria20}. Namely, the criteria in Theorem \ref{thm2.1} hold with the matrix $D(t)=\diag \,  (d_1(t),\dots,d_n(t))$ replaced by $\tilde D(t)=[\tilde d_{ij}(t)]$, where $\tilde d_{i}(t)=d_i(t)$ and $\tilde d_{ij}(t)=-d_{ij}(t)$ for $j\ne i$. Naturally, in this case,  the condition $\liminf_{t\to\infty}d_{i}(t)>0$  in (ii) of  Theorem \ref{thm2.1} should be replaced by
$\liminf_{t\to\infty}(d_{i}(t)v_i-\sum_jd_{ij}(t)v_j)>0$, for all $i$.  This means that the criterion for permanence in Theorem \ref{thm2.3} remains valid with these changes.}\end{rmk}

\begin{rmk}\label{rmk2.6} {\rm Consider nonlinearites   which also incorporate a strictly sublinear negative feedback term of the form $-K_i(t,x_i(t))$, so that  \eqref{NonLin} reads as
\begin{equation}\label{NonLin1}
\begin{split}
x_i'(t)&=-d_i(t)x_i(t)+\sum_{j=1}^nL_{ij}(t)x_{j,t}+f_i(t,x_{i,t})-K_i(t,x_i(t)), \ t\ge 0,  i=1,\dots,n,
\end{split}
\end{equation}
where $K_i(t,x)\ge 0$ are continuous and
$K_i(t,x)\le \ka_i(t)g_i(x)$
for some continuous functions $\ka_i,g_i:\R^+\to \R^+$ with $\ka_i(t)$ bounded, $g_i(0)=0$ and with right-hand derivative $(g_i)'(0+)=0$. With $f_i$ bounded functions,
solutions of \eqref{NonLin1} satisfy the inequalities  $-d_i(t)x_i(t)-K_i(t,x_i(t))\le x_i'(t)\le -d_i(t)x_i(t)+\sum_{j=1}^n L_{ij}(t)x_{j,t}+C\, (1\le i\le n)$, where $C>0$ is such that $f_i(t,\phi)\le C$ on $\R^+\times C^+$ for all $i$. By comparing  below and  above the solutions of \eqref{NonLin1} with solutions of cooperative systems and from Theorem \ref{thm2.2}, it follows  \eqref{NonLin1} is dissipative and  that $C^+_0$ is forward invariant for \eqref{NonLin1}.
On the other hand, for any fixed $\vare>0$ small, there is $m_0>0$ such that $0\le K_i(t,x)\le \vare x$ for $x\in [0,m_0]$.  
A careful analysis  shows that  the arguments in the proof of Theorem \ref{thm2.3} carry over to \eqref{NonLin1} if one chooses $\vare \in (0,\de)$, for $\de>0$ as in (H5), so that \eqref{h5} is satisfied with $d_i(t)$ replaced by $d_i(t)+\vare$. In this way, one may conclude
that the permanence results stated in Theorems \ref{thm2.3},  \ref{thm2.3'} and Corollary \ref{cor2.1} are still  valid for \eqref{NonLin1}. This more general framework allows in particular to consider structured models with harvesting.}
\end{rmk}

\section{Applications and examples}
\setcounter{equation}{0}

We now apply our results to generalized Nicholson  and Mackey-Glass systems. The literature on generalized Nicholson  and Mackey-Glass models  is very extensive, here we only mention a few selected references dealing  with the persistence and permanence  for either scalar or multidimensional Nicholson equations  \cite{BIT,FOS,FariaRost,Liu10,ObayaSanz} and  Mackey-Glass equations \cite{BB17,BBI,FOS}, and references therein.

Consider systems given by
 \begin{equation}\label{N}
\begin{split}
x_i'(t)&=-d_i(t)x_i(t)+\sum_{j=1}^n L_{ij}(t)x_{j,t} \\
&+\sum_{k=1}^{m_i} b_{ik}(t)  \int_{t-\tau_{ik}(t)}^t\!\!  \la_{ik}(s)g_{ik}(s,x_i(s))\, ds,\ i=1,\dots,n,
\end{split}
\end{equation}
where $L_{ij}(t)x_{j,t} =a_{ij}(t) \int_{-\sigma_i(t)}^0 x_j(t+s)\, d_s\nu_{ij}(t,s)$ are as in \eqref{Lij}  with $ \int_{-\sigma_i(t)}^0 d_s\nu_{ij}(t,s)=1$ and either
\begin{equation}\label{N-MG1}
 g_{ik}(t,x)=x\e^{-c_{ik}(s)x}
 \end{equation}
 or \begin{equation}\label{N-MG2}
  g_{ik}(t,x)=\frac{x}{1+c_{ik}(t)x^{\al_i}}\q (\al_i\ge 1).
 \end{equation}
 The functions $d_i(t),a_{ij}(t),b_{ik}(t),\sigma_i(t),\tau_{ik}(t)$, $\la_{ik}(t),c_{ik}(t)$ are assumed to be continuous and nonnegative, with $\sigma_i(t),\tau_{ik}(t)\in [0,\tau]$ (for some $\tau>0$), and $d_i(t)>0, c_{ik}(t)>0$, 
 for all $i,j,k$ and $t\ge 0$.
For $g_{ik}$ as in \eqref{N-MG1} a modified
  Nicholson-type system is obtained,  whereas the choice \eqref{N-MG2} provides a  Mackey-Glass-type system.
We suppose  that the linear operators  $L_{ij}(t)$ are  {\it nonnegative}, thus $a_{ij}(t)=\| L_{ij}(t)\|$ as before, and define
$$\be_i(t):=\sum_{k=1}^{m_i} b_{ik}(t)\int_{t-\tau_{ik}(t)}^t\!\!  \la_{ik}(s)\, ds>0,\q t\ge 0\q i=1,\dots,n.$$

From Theorem \ref{thm2.3}, we  derive sufficient conditions for the permanence of \eqref{N}.

  \begin{thm}\label{thm3.1} For  \eqref{N} under the general conditions above, let the matrices $D(t),A(t),M(t)$ be as in \eqref{D&A},  \eqref{M(t)},  and assume that: \vskip 0mm
  (i) the functions  $a_{ij}(t),c_{ik}(t),\be_i(t)$ are bounded  on $\R^+$ for all $ i,j=1,\dots,n,k=1,\dots,m_i$;\vskip 0mm
 (ii)  there are positive vectors $u,v$ and  $\de >0$ such that
$[D(t)- A(t)-\de I]u\ge 0,\ [M(t)-\de I]v\ge 0.$\\
Then   system \eqref{N} is permanent.\end{thm}

\begin{proof} Clearly, system \eqref{N} has the form \eqref{NonLin}, with
 $f_i(t,\phi_i)=\sum_{k=1}^{m_i} b_{ik}(t)  \int_{-\tau_{ik}(t)}^0\!\!  \la_{ik}(t+s)g_{ik}(t+s,\phi_i(s))\, ds$. Since $g_{ik}$ and $\be_i(t)$ are bounded, from Theorem \ref{thm2.2} the system is dissipative. Choose $M>0$ such that
 $0<x_i(t)\le M,\, t\in\R^+, i=1,\dots,n$, for any solution $x(t)$ with initial condition in $C_0^+$.

Let
$0<c_{ik}(t)\le \ol{c_i} $ for $t\ge 0$
and all $i,k$.  Define $h_i(x):=x\e^{- \ol{c_i}x}$, respectively $h_i(x):=\frac{x}{1+\ol{c_i}x^{\al_i}}$, for  Nicholson, respectively Machey-Glass systems. We always have $h_i'(0)=1$. Moreover, the functions $h_i$ are unimodal with  $h_i(\infty)=0$, with the exception of the increasing and bounded function $\frac{x}{1+\ol{c_i}x}$ (when $\al_i=1$ in \eqref{N-MG2}). Now, we reason as in the proof of Theorem \ref{thm2.3}. Choose $m>0$ such that all the functions $h_i(x)$ are increasing in $[0,m]$ and $h_i(m)\le h_i(M)$, and define  $h_i^-(x):=\begin{cases} h_i(x),\ 0\le x\le m\\ h_i(m),\ x>m\end{cases}$.
 The above conditions imply that hypotheses  (H1), (H3), (H4) are satisfied, for $\be_i(t)$ as above and these choices  of $h_i^-(x)$. From (iii), (H2) and (H5) are satisfied. The result follows from Theorem \ref{thm2.3}.
\end{proof}

 \begin{rmk}\label{rmk3.1} {\rm  As mentioned previously, Theorems \ref{thm2.3} and \ref{thm3.1} are still valid if one replaces (H2)   by the assumptions (H2*) and $\liminf_{t\to\infty}d_{i}(t)>0$, for all $i$.}
\end{rmk}

For the situation without delays in the linear part, from Corollary \ref{cor2.1} we obtain: 

\begin{cor}\label{cor3.0} For $g_{ik}$ as in \eqref{N-MG1} or \eqref{N-MG2}, consider the system
\begin{equation}\label{NFOS}
\begin{split}
x_i'(t)&=-d_i(t)x_i(t)+\sum_{j=1}^n a_{ij}(t)x_j(t)
+\sum_{k=1}^{m_i} b_{ik}(t)  \int_{t-\tau_{ik}(t)}^t\!\!  \la_{ik}(s)g_{ik}(s,x_i(s))\, ds, \ i=1,\dots,n,
\end{split}
\end{equation}
under the above conditions on the coefficients and delays,  and assume that: \vskip 0mm
(i) the functions  $c_{ik}(t),\be_i(t)$ are bounded  on $\R^+$  for all $ i=1,\dots,n,k=1,\dots,m_i$;\vskip 0mm
 (ii)  there are positive vectors $u,v$ and  $\de >0,\al>1$ such that $[D(t)-A(t)-\de I]u\ge 0,B(t)v\ge \al [D(t)-A(t)]v.$\\
 Then     \eqref{NFOS}  is permanent.\end{cor}

We  emphasize that this corollary  gives a sharper criterion for permanence than the one in  \cite{FOS}, and moreover applies to a much larger family of delayed structured models. For instance, in the case of Nicholson systems, the result in Corollary \ref{cor3.0} was established in \cite[Theorem 3.5]{FOS} only for the case of Nicholson systems \eqref{NFOS} with {\it discrete delays} and all coefficients {\it bounded}.

 Some illustrative examples, as well as counter-examples showing the necessity of our assumptions, are now presented.

\begin{exmp}\label{exmp2.3}  {\rm This counter-example is based on a counter-example due to Gy\H{o}ri and Horv\'ath \cite{GH17}, and shows that if (H2) holds but the coefficients $a_{ij}(t)$ are not bounded, then even  the asymptotic stability of  \eqref{Lin0} may fail.

Consider a planar linear DDE of the form
\begin{equation}\label{MG3_Lin}
 \begin{split}
x_1'(t)&=-d(t) x_1(t)+a(t) x_2(t-\tau(t))\\
x_2'(t)&=-d(t) x_2(t)+a(t) x_1(t-\tau(t))\end{split}\q ,\ t\ge 0,
\end{equation}
and the  scalar equation
\begin{equation}\label{GH_Lin}
 y'(t)=-d(t) y(t)+a(t) y(t-\tau(t)),\q t\ge 0,
\end{equation}
where $d(t),a(t),\tau(t)$ are continuous and positive and $\tau(t)$ is bounded,  for $t\in\R^+$. It is clear that if $y(t)$ is a solution of  \eqref{GH_Lin},  then $x(t)=(y(t),y(t))$ is a solution of  \eqref{MG3_Lin}.


Take $C>\tau$. Following the example in \cite[Proposition 1]{GH17},  choosing $d(t)$ such that
$$
d(t)=\frac{1}{(t+C)(t+1+C)}+a(t)\frac{(t-\tau(t)+C+1)(t+C)}{(t-\tau(t)+C)(t+C+1)},
$$
then
$$\var(t)=1+\frac{1}{t+C}$$
is a solution of the linear equation \eqref{GH_Lin}.
On the other hand, since
$$d(t)-a(t)=\frac{1}{(t+C)(t+1+C)}+a(t) \frac{\tau(t)}{(t-\tau(t)+C)(t+C+1)},$$
if $\mu>0$ is fixed and $a(t)$ is chosen to be
$a(t)=\mu \frac{(t-\tau(t)+C)(t+C+1)}{\tau(t)},$
we have $a(t)\to \infty$ as $t\to \infty$ and $d(t)-a(t)=\frac{1}{(t+C)(t+1+C)}+ \mu\ge \mu$. With our previous notations, for  \eqref{MG3_Lin} we have
$$D(t)=\diag (d(t),d(t)),\q A(t)=\left[ \begin{array}{cc}0&a(t) \\
a(t)&0\end{array}\right],$$ thus $[D(t)-A(t)-\mu I]\vec 1>0$, and (H2) is satisfied. However, (H2*) does not hold. Since  \eqref{MG3_Lin} possesses a solution $(\var(t),\var(t))\to (1,1)$ as $t\to \infty$, the system is not asymptotically stable. }
\end{exmp}

\begin{exmp}\label{exmp2.1}  {\rm Consider the Mackey-Glass-type  system
 \begin{equation}\label{MG1}
 \begin{split}
x_i'(t)&=-d_{ii}t^\eta x_i(t)+\sum_{ j=1,j\ne i}^nd_{ij}t^\eta x_j(t)+\sum_{j=1}^nb_{ij}t^\eta\int_{-\tau_{ij}(t)}^0 x_j(t+s)\, ds\\
&+\be_i(t)\int_{-\sigma_i(t)}^{0}k_i(s)\frac{x_i(t+s)}{1+c_i(t)x_i(t+s)^{\nu_i}}ds,\q 1\le i\le n,
\end{split}
\end{equation}
where $\eta>0,\nu_i>0, b_{ij}, d_{ij}\in \R^+$ with $d_i:=d_{ii}>0$ for all $i$, and the delays $\tau_{ij}(t), \sigma_i(t)$ are continuous with  $0\le \tau_{ij}(t)\le r_{ij}, 0<\sigma_i(t)\le R_i$ for some constants $r_{ij},R_i>0$, $\be_i,c_i:\R^+\to (0,\infty)$ are continuous, $c_i(t)$ are bounded, and $k_i:\R^-\to\R^+$ are integrable with $\int_{-\sigma_i(t)}^0k_i(s)\, ds=1$,   $t\ge 0,i,j=1,\dots,n$.  Clearly,  (H4) is satisfied with $h_i^-(x)=\frac{x}{x+\ol{c}_ix^{\nu_i}}$, for $\ol{c}_i>0$ such that  $c_i(t)\le \ol{c}_i\, (1\le i\le n)$.

  With the previous  notation, we have $d_{i}(t)=d_{i}t^\eta, a_{ij}(t)=((1-\de_{ij})d_{ij}+b_{ij}\tau_{ij}(t))t^\eta \le [(1-\de_{ij})d_{ij}+b_{ij}r_{ij}] t^\eta$, where $\de_{ij}=1$ if $i=j$ and  $\de_{ij}=0$ if $i\ne j$. Define the $n\times n$ matrices $D=\diag\, (d_1,\dots,d_n), A=\Big [(1-\de_{ij}) d_{ij}+b_{ij}r_{ij}\Big ]$, so that $D(t)=t^\eta D, A(t)=t^\eta A$. Assume that
$$N:=D- A$$
 is a {\it non-singular M-matrix}; or, in other words, that there exists a positive vector $v$ such that $u:=Nv>0$. For $\de>0$ small such that $\de Av\le u$, we have $Dv\ge (1+\de)Av$,
 hence  (H2*) is satisfied.  From Theorem \ref{thm2.1}(ii) we deduce that the linear system
 $$x_i'(t)=-d_{ii}t^\eta x_i(t)+\sum_{ j=1,j\ne i}^nd_{ij}t^\eta x_j(t)+\sum_{j=1}^nb_{ij}t^\eta\int_{-\tau_{ij}(t)}^0 x_j(t+s)\, ds,\q i=1,\dots,n,$$ is exponentially asymptotically stable. Note however that none of its coefficients is  bounded on $\R^+$.
 
 Next, suppose that $0<\nu_i\le 1$ for all $i$, which implies that $\lim_{x\to\infty}h_i^-(x)=\infty$ if  $0<\nu_i< 1$
 and $\lim_{x\to\infty}h_i^-(x)=1$ if  $\nu_i=1$. If there is $\al >1$ such that $\be_i(t)v_i\ge \al t^\eta u_i$ for $t\gg 1$, then also $\be_i(t)v_i\ge t^\eta u_i+\de $, for some $\de>0$ and for $t\gg 1$,  $1\le i\le n$. Under this condition, both (H5) and
 (H5*) hold. From Theorem \ref{thm2.4} we conclude that \eqref{MG1} is uniformly persistent.
}
\end{exmp}

\begin{exmp}\label{exmp2.2}  {\rm Consider the  planar  system 
 \begin{equation}\label{MG2}
 \begin{split}
x_1'(t)&=-t^\eta x_1(t)+(t^\eta-1) x_2(t-\tau_1(t))
+\be h_1(t,x_1(t-\sigma_1(t))\\
x_2'(t)&=-t^\eta x_2(t)+(t^\eta-1) x_1(t-\tau_2(t))
+\be h_2(t,x_2(t-\sigma_2(t))\end{split}\q ,\ t\ge 1,
\end{equation}
 with nonlinearities of either Mackey-Glass or Nicholson type, $$h_i(t,x)=\frac{x}{1+c_i(t)x^{\nu_i}}\ {\rm with} \ \nu_i\ge 1,\q {\rm or}\q
h_i(t,x)=x\e^{-c_i(t)x},
$$ and $\eta>0,\be >1$, where the delays $\tau_i(t), \sigma_i(t)$ are  nonnegative, continuous and bounded,  $c_i(t)$ are positive, continuous and   $0<c_i(t)\le \ol{c}_i$ for some constants $\ol{c}_i,\, i=1,2$.  
 With  the previous notation, $d_i(t)=t^\eta, a_{ii}(t)=0, \be_i(t)\equiv\be>1,\, i=1,2$ and $a_{12}(t)=a_{21}(t)=t^\eta -1$, thus 
$$D(t)=\diag (t^\eta,t^\eta),\q A(t)=\left[ \begin{array}{cc}0&t^\eta-1 \\
t^\eta-1&0\end{array}\right], \q M(t)=\left[ \begin{array}{cc}\be-t^\eta&t^\eta-1 \\
t^\eta-1&\be-t^\eta\end{array}\right].$$
As
$$[D(t)-A(t)]\left[ \begin{array}{cc}1 \\
1\end{array}\right]=\left[ \begin{array}{cc}1 \\
1\end{array}\right],\q M(t)\left[ \begin{array}{cc}1 \\
1\end{array}\right]=\left[ \begin{array}{cc}\be-1 \\
\be -1\end{array}\right],$$
 (H2), (H5) (and thus also (H5*)) are satisfied. Since $a_{12}(t),a_{21}(t)$ are not bounded, we cannot deduce that \eqref{MG2} is permanent (nor that (H2*) is satisfied). However, if there are no delays in the linear part, i.e., $\tau_i(t)=\tau_2(t)\equiv 0$ in  \eqref{MG2}, from Theorem \ref{thm2.1}(iii) we deduce that the linear ODE
$
 \begin{cases}
x_1'(t)=-t^\eta x_1(t)+(t^\eta-1) x_2(t)\\
x_2'(t)=-t^\eta x_2(t)+(t^\eta-1) x_1(t)\end{cases}$
is exponentially asymptotically stable. The permanence of \eqref{MG2} follows then from Corollary \ref{cor2.1}.}
\end{exmp}

\begin{exmp}\label{exmp2.4}  {\rm Consider the scalar equation
 \begin{equation}\label{exmp4}x'(t)=-d(t)x(t)+a(t)x(t-\tau(t))+\be (t)h(x(t)),\q t\ge 0,\end{equation}
where $\tau(t)$ is continuous with $0\le \tau(t)\le \tau$, $h(x)= \begin{cases} x^2\ {\rm if}\ 0\le x\le 1\\ 1\ {\rm if}\ x\ge 1\end{cases}$ and
$$a(t)=\mu \frac{t+C-\tau(t)}{\tau(t)},\ \be(t)=\mu_1\frac{t+C}{t+C-1}, \ d(t)=a(t)+(\be(t)+1)\frac{1}{t+C}+\mu.$$
for some constants $C> \max (\tau,1), \mu>0, \mu_1>\mu +1/C$. One easily sees that this equation has the solution $\var(t)=\frac{1}{t+C}$, thus \eqref{exmp4} is not persistence.  Note that $\be(t)$ is bounded and  $d(t)-a(t)>\mu, \be(t)-d(t)+a(t)=\mu_1-\mu-\frac{1}{t+C}\ge \mu_1-\mu-\frac{1}{C}>0$, therefore (H2), (H5) and (H5*) are satisfied.  However, hypothesis (H4) is not fulfilled, because $h'(0)=0$ (conf. Remark \ref{rmk2.3}).}\end{exmp}

\begin{exmp}\label{exmp2.5}  {\rm Consider the system
 \begin{equation}\label{exmp5}
 \begin{split}x_1'(t)&=-(a(t)+d_1(t))x_1(t)+a(t)x_2(t)+\be (t)h(x_1(t-\tau)),\\
  x_2'(t)&=-(a(t)+d_1(t))x_2(t)+a(t)x_1(t)+\be (t)h(x_2(t-\tau)),\q t\ge 0,
 \end{split}\end{equation}
where $\tau>0$, $a(t)$ is continuous, nonnegative and bounded, $h(x)=\frac{x}{1+x}$ and
$$d_1(t)=\mu \frac{t+C}{1-\tau}, \q \be(t)=\frac{t+C+1-\tau}{t+C}\Big (d_1(t)-\frac{1}{t+C}\Big),$$
for some $\tau\in (0,1)$ and $C>\tau$.
It is easy to verify that 
$x(t)=\Big(\frac{1}{t+C},\frac{1}{t+C}\Big)$
is a solution of \eqref{exmp5}, thus this system is not permanent.
Note that $d_1(t)\ge \mu \frac{C}{1-\tau}, \be(t)-d_1(t)=\mu-\frac{t+C+1-\tau}{(t+C)^2}\ge \mu/2$ for $t\gg 1$,
hence (H2) and (H5) are satisfied. But Theorem \ref{thm2.3} does not apply, because $\be(t)$ is not bounded.
}\end{exmp}

\section{Discussion and open problems}
\setcounter{equation}{0}

In this paper, we have proven the permanence of delayed differential systems \eqref{NonLin} which  incorporate distributed delays in both the linear and nonlinear parts and are in general noncooperative. Moreover, not all the coefficients are required to be bounded. The main  Theorem \ref{thm2.3} extends known results in recent literature \cite{BBI,BIT,FOS,FariaRost, GHM18,GHM18a,Liu10}, as it applies to a broad family of nonautonomous delay differential systems. 

Once the permanence of \eqref{NonLin} is guaranteed, several open questions arise and should be addressed. First, it would be interesting to have explicit  lower and upper uniform bounds for all positive solutions, as investigated  in \cite{BIT,FariaAMC14,FariaJDDE16,GHM15,GHM18, GHM18a} for cooperative scalar or $n$-dimensional DDEs and in \cite{FariaRost,Liu10} for  noncooperative systems. Secondly,  the global stability of DDEs is a matter of  crucial importance in applications, therefore a relevant  task is to  propose sufficient conditions forcing $x(t)-y(t)\to 0$ as $t\to\infty$, for  any  two positive solutions $x(t),y(t)$ of \eqref{NonLin}. In the case of nonautonomous noncooperative models, it is however clear that  the response to these two questions  depends on the specific nonlinearities. In a forthcoming paper, 
  these topics will be addressed for generalized Nicholson systems. For periodic $n$-dimensional DDEs, it has been proven \cite{Zhao} that in some settings  the permanence implies the existence of a positive  periodic solution -- in this context, a stability result will show that such a periodic solution is a global attractor of all positive solutions.

It is worthwhile mentioning that, in  the last few years,   the     stability of nonautonomous linear 
 DDEs has received a great deal of attention,  and several methods have been used to obtain  explicit sufficient conditions for  the asymptotic and exponential asymptotic stability of a general linear system \eqref{Lin0}, see e.g.~\cite{BDSS18,Faria20,GH17} and references therein. Actually,  both delay independent and delay-dependent  criteria for the stability of linear DDEs with possible {\it infinite} delays were given in \cite{Faria20}. Since the exponential stability of \eqref{Lin0} is a key ingredient to show the permanence of   \eqref{NonLin}, this leads us to two natural lines of future research, explained below.
 
 The first one is to replace assumption (H2) or (H2*) -- which forces \eqref{Lin0} to possess   diagonal  terms without delay which dominate the effect of the delayed terms -- by a condition depending on the size of delays in such a way that  \eqref{Lin0} maintains the exponential asymptotic stability, and further analyse how such a condition interplays with the assumption (H5). 
 
 Another open problem is to study the persistence and permanence of systems of the form \eqref{NonLin} with {\it unbounded} delays. DDEs with infinite delay are surely more challenging: not only an {\it admissible} phase space  satisfying some fundamental set of axioms should be chosen \cite{HMN}, but most techniques for finite delays do not apply for such equations. 
 There has been some work on permanence for scalar nonautonomous DDEs with infinite delay, see e.g.~\cite{GH17}. In the case of multidimensional  DDEs with infinite delay, the work in \cite{FariaJDDE16}  only  contemplates situations of cooperative systems, namely of the form $x_i'(t)=F_i(t,x_t)-x_i(t)G_i(x_t)\, (1\le i\le n)$ with $F_i,G_i$ cooperative and $F_i(t,x)$ sublinear in $x\in \R^+$. For the case of nonmonotone nonlinearities in  \eqref{NonLin}, it is clear that the technique developed in the proof of Theorem \ref{thm2.3} does not apply to systems with infinite delay, since it relies on a step-wise iterative argument on intervals of lenght $\tau$, where $\tau$ is the supremum of all delays -- thus, new tools and arguments to tackle the difficulty must be proposed. This open problem  is a strong motivation for a next future investigation.


The treatment of  mixed monotonicity models, in what concerns questions of  permanence, is another topic deserving  attention, since they appear naturally in real-world applications.  In fact, there has been an increasing interest in DDEs with mixed monotonicity, where  the nonlinear terms involve one or more  functions with  different   delays e.g.~of the form $f(t,x(t-\tau(t)),x(t-\sigma(t)))$, with $f(t,x,y)$  monotone increasing in the variable $x$ and monotone decreasing in $y$. 
As as illustrated by Berezansky and Braveman \cite{BB17},  though  small delays are in general harmless, the presence of  two or more delays in the same nonlinear function may change drastically the  global  properties of the solutions. The
 permanence and stability of DDEs with nonlinearities of mixed monotonicity have been analyzed in \cite{BB16, BB17,ElRuiz,GHM15,GHM18a}. As far as the author knows,  only the case of  discrete delays has been dealt with. As seen,  systems \eqref{NonLin} encompass models with noncooperative nonlinearities, nevertheless cooperative techniques were used in our arguments. Therefore, new tools are required to handle the case of  mixed monotonicity in the nonlinear terms.
 
\section*{Acknowledgements}
This work was  supported by National Funding from FCT - Funda\c c\~ao para a Ci\^encia e a Tecnologia (Portugal) under project UIDB/04561/2020.

\end{document}